\documentclass[12pt]{article}
\usepackage{geometry}                % See geometry.pdf to learn the layout options. There are lots.
\usepackage{graphicx}
\usepackage{amssymb, mathrsfs, enumerate}

\usepackage{amscd}
\usepackage{amssymb}
\usepackage{graphics}
\usepackage{amsmath}
\usepackage[english]{babel}
\usepackage{hyperref}
\usepackage[latin1]{inputenc}

\usepackage{color}      % Need the color package
\usepackage{srcltx}  % more source specials

\usepackage{amsthm, amssymb}
\usepackage{setspace}
\usepackage[all]{xy}
\usepackage{graphicx}
\usepackage{ifpdf}
\usepackage{babel}
\usepackage{verbatim}

\newcommand{\BZ}{{\mathbb{Z}}}
\newcommand{\BN}{{\mathbb{N}}}
\newcommand{\BR}{{\mathbb{R}}}

\newcommand{\BQ}{{\mathbb{Q}}}

\newcommand{\QQ}{{\mathbf{Q}}}
\newcommand{\RR}{{\mathbf{R}}}

\newcommand{\gd}{\delta}

\newcommand{\gC}{\Gamma}

\newcommand{\gS}{\Sigma}
\newcommand{\gO}{\Omega}

\newcommand{\gep}{\epsilon}

\newcommand{\ga}{\alpha}

\newcommand{\Ad}{\mathrm{Ad}}

\newcommand{\vol}{\mathrm{vol}}

\newcommand{\SL}{\mathrm{SL}}

\newcommand{\PSL}{\mathrm{PSL}}

\newcommand{\SO}{\mathrm{SO}}

\newcommand{\irs}{\mathrm{IRS}}
\newcommand{\sub}{\mathrm{Sub}}

\newcommand{\act}{\curvearrowright}

\newtheorem{prop}{Proposition}[section]
\newtheorem{thm}[prop]{Theorem}

\newtheorem{cor}[prop]{Corollary}

\newtheorem*{conj*}{Conjecture}
\theoremstyle{definition}

\newtheorem{defn}[prop]{Definition}
\newtheorem{rem}[prop]{Remark}
\newtheorem*{rem*}{Remark}

\newtheorem{exam}[prop]{Example}
\newtheorem{prob}[prop]{Problem}
\newtheorem{ques}[prop]{Question}

\newtheorem{theorem}{Theorem}

\newtheorem{definition}[theorem]{Definition}

\newtheorem{exercise}[theorem]{Exercise}

\newcommand{\blue}[1]{{\color{blue}{#1}}}

\title{Lecture notes on Invariant Random Subgroups and Lattices in rank one and higher rank}
\author{Tsachik Gelander\thanks{I would like to thank Tal Cohen, Gil Goffer, Arie Levit and the referee for corrections and valuable advices concerning the preliminary version.}}

\begin{document}
\maketitle

Invariant random subgroups (IRS) are conjugacy invariant probability measures on the space of subgroups in a given group $G$. They can be regarded both as a generalization of normal subgroups as well as a generalization of lattices. As such, it is intriguing to extend results from the theories of normal subgroups and of lattices to the context of IRS. Another approach is to analyse and then use the space IRS$(G)$ as a compact $G$-space in order to establish new results about lattices. The second approach has been taken in the work \cite{7},  that came to be known as the seven samurai paper.  
In these lecture notes we shall try to give a taste of both approaches.

\section{The Chabauty space of closed subgroups}
Let $G$ be a locally compact group. We denote by $\text{Sub}_G$ the space of closed subgroups of $G$ equipped with the Chabauty topology. This topology is generated by sets of the following two types: 
\begin{enumerate}
\item $O_1(U):=\{ H\in\text{Sub}_G:H\cap U\ne\emptyset\}$ with $U\subset G$ an open subset, and
\item $O_2(K):=\{ H\in\text{Sub}_G:H\cap K=\emptyset\}$ with $K\subset G$ a compact subset.
\end{enumerate}

\begin{exercise}
Show that a sequence $H_n\in\text{Sub}_G$ converges to a limit $H$ iff
\begin{itemize}
\item for any $h\in H$ there is a sequence $h_n\in H_n$ such that $h=\lim h_n$, and
\item for any sequence $h_{n_k}\in H_{n_k}$, with $n_{k+1}>n_k$, which converges to a limit, we have $\lim h_{n_k}\in H$.
\end{itemize}
\end{exercise}

\begin{exercise}[Suggested by Ian Biringer]
The space $\text{Sub}_G$ is metrisable. Indeed, let $d$ be a proper metric on $G$ and $d_H$ the corresponding Hausdorff distance between compact subsets of $G$. Show that
$$
 \rho(H_1,H_2):=\int_0^\infty d_\text{H}\big( H_1\cap B_r(id_G),H_2\cap B_r(id_G)\big)e^{-r}dr
$$
is a compatible metric on $\text{Sub}_G$.
\end{exercise}

\begin{exam} 
\begin{enumerate}
\item $\text{Sub}_\BR\sim [0,\infty ]$. Indeed, every proper non-trivial closed subgroup of $\BR$ is of the form $\ga\BZ$ for some $\ga>0$. When $\ga\to 0$ the corresponding group tends to $\BR$ and when $\ga\to \infty$ it tends to $\{ 1\}$.

\item $\text{Sub}_{\BR^2}$ is homeomorphic to the sphere $S^4$ (this was proved by Hubbard and Pourezza, see \cite{HP}).

\item {\it Question:} What can you say about $\text{Sub}_{\BR^n}$ (cf. \cite{Kl})?
\end{enumerate}
\end{exam}

One direction which seems interesting to study
is the case of semisimple Lie groups. Indeed, much is known about discrete and general closed subgroups of (semisimple) Lie groups, and it is possible to deduce information about the structure of $\text{Sub}_G$. 

\begin{prob}
What can you say about $\text{Sub}_G$ for $G=\SL_2(\BR)$?
\end{prob}

\subsection{Compactness}

While the structure of $\text{Sub}_G$ in general is highly complicated, we at least know that it is always compact:

\begin{prop}(Exercise)
For every locally compact group $G$, the space $\text{Sub}_G$ is compact.
\end{prop}

We can use $\text{Sub}_G$ in order to compactify certain sets of closed subgroups. For instance one can study the Chabauty compactification of the space of lattices in $G$. In particular, it is interesting to determine the points of that compactification:

\begin{prob}
Determine which subgroups of $\SL_3(\BR)$ are limit of lattices.
\end{prob}

This problem might be more accessible if we replace $\SL_3(\BR)$ with a group for which the congruence subgroup property is known for all lattices.

\subsection{When is $G$ isolated?}

%Here is one simple, yet useful, result:  

It is useful to know under which conditions $G$ is an isolated point in $\text{Sub}_G$.

\begin{exercise}
A discrete group $\Gamma$ is isolated (as a point in) in $\text{Sub}_\Gamma$ iff it is finitely generated.
\end{exercise}

Let us examine some non-discrete cases: 

\begin{prop}\label{prop:G-isolated}
Let $G$ be a connected simple Lie group. Then $G$ is an isolated point in $\text{Sub}_G$.
\end{prop}

The idea is that if $H$ is sufficiently close to $G$ then it has points close to $1$ whose logarithm generate the Lie algebra of $G$. This implies that the connected component of identity $H^\circ$ is normal in $G$, and as $G$ is simple, it is either trivial or everything. Thus, it is enough to show that $H^\circ$ is non-trivial, i.e. that $H$ is not discrete. This is proved in \cite{To,Ku} --- more precisely it is shown there that a non-nilpotent connected Lie group is never a limit of discrete subgroups. 

\begin{exercise} 
Show that $G$ is not a limit of discrete subgroups, relying on the classical:

\begin{thm}[Zassenhaus, see \cite{Rag}]\label{Zass}
A Lie group $G$ admits an identity neighborhood $U$ such that for every discrete group $\Gamma\le G$, $\langle \log(\Gamma\cap U)\rangle$ is a nilpotent Lie algebra.
%Every Lie group $G$ admits an identity neighbourhood $U$ such that if $H\le G$ is a discrete group generated by $H\cap U$ then $H$ is contained in a connected nilpotent subgroup of $G$.
\end{thm}
\end{exercise}

For more details about Proposition \ref{prop:G-isolated} as well as other results in this spirit, see \cite[Section 2]{7}.

\begin{exam}
The circle group $S^1$ is not isolated in its Chabauty space $\text{Sub}_{S^1}$.  Indeed, one can approximate $S^1$ by finite cyclic subgroups. 
\end{exam}

More generally,

\begin{exercise}
Show that if $G$ surjects on $S^1$, then $G$ is not isolated in $\text{Sub}_G$.
\end{exercise}

In fact, using Theorem \ref{Zass} one can show:

\begin{prop}
A connected Lie group is isolated in $\text{Sub}_G$ iff it does not surject on $S^1$ (i.e. has no non-trivial characters). 
\end{prop}

Note that a connected Lie group $G$ does not surject on the circle iff its commutator $G'$ is dense in $G$. Such groups are called {\it topologically perfect}.

\begin{exercise}
Let $G$ be a Lie group.
Deduce from Theorem \ref{Zass} that if $H\in \text{Sub}_G$ is a limit of discrete groups, then the identity connected component $H^\circ$ of $H$ is nilpotent.
\end{exercise}

A similar result can be proved for semisimple groups over non-Archimedean local fields:

\begin{exercise}
Consider $G=\SL_n(\BQ_p)$ and show that $G$ is an isolated point in $\text{Sub}_G$.
\end{exercise}

\noindent
{\it Hint:} Use the following facts:
\begin{itemize}
\item $\SL_n(\BZ_p)$ is a maximal subgroup of $\SL_n(\BQ_p)$.
\item The Frattini subgroup of $\SL_n(\BZ_p)$ is open, i.e. of finite index.
\end{itemize}

The following result is proved in \cite{GL}.

\begin{thm}
Let $G$ be a semisimple analytic group over a local field $k$, then $G$ is isolated in $\text{Sub}_G$.
\end{thm}

\section{Invariant measures on $\text{Sub}_G$}

The group $G$ acts on $\text{Sub}_G$ by conjugation and it is natural to consider the invariant measures on this compact $G$-space. 
%This brings us to the main definition of this lecture:

\begin{defn}
An Invariant Random Subgroup (hereafter IRS) is a Borel regular probability measure on $\text{Sub}_G$ which is invariant under conjugations.
\end{defn}

\subsection{First examples and remarks:}\label{subsection:examples}

\medskip
\noindent 
$(1)$ The Dirac measures correspond to normal subgroups. In view of this, one can regard IRS's as a generalization of normal subgroups.

\medskip
\noindent
$(2)$ Let $\gC\le G$ be a lattice (or more generally a closed subgroup of finite co-volume). The map $G\to \gC^G\subset\text{Sub}_G,~g\mapsto g\gC g^{-1}$, factors through $G/\gC$. Hence we may push the invariant probability measure on $G/\gC$ to a conjugation invariant probability measure on $\text{Sub}_G$ supported on (the closure of) the conjugacy class of $\Gamma$.  In view of that  IRS's also generalize `lattices' or more precisely finite volume homogeneous spaces $G/\gC$ --- as conjugated lattices give rise to the same IRS.
We shall denote the IRS associated with (the conjugacy class of) $\gC$ by $\mu_\gC$.

\medskip

For instance let $\Sigma$ be a closed hyperbolic surface and normalize its Riemannian measure. Every unit tangent vector yields an embedding of $\pi_1(\Sigma)$ in $PSL_2(\RR)$. Thus the probability measure on the unit tangent bundle corresponds to an IRS of type (2) above.  

\begin{center}
\includegraphics[height=4cm]{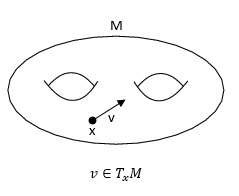}
\end{center} 

%For instance let $\gS$ be a closed hyperbolic surface (of genus $\ge 2$) and normalize its Riemannian measure. Every unit tangent vector yields an embedding of $\pi_1(\gS)$ in $\PSL_2(\BR)$. Thus the probability measure on the unit tangent bundle corresponds to an IRS of type (2) above.  

\medskip
\noindent
$(3)$ Let again $\gC\le G$ be a lattice in $G$, and let $N\lhd\gC$ be a normal subgroup of $\gC$. As in $(2)$ the $G$-invariant probability measure on $G/\gC$ can be used to choose a random conjugate of $N$ in $G$. This is an IRS supported on the (closure of the) conjugacy class of $N$. More generally, every IRS on $\gC$ can be induced to an IRS on $G$. Intuitively, the random subgroup is obtained by conjugating $\gC$ by a random element from $G/\gC$ and then picking a random  subgroup in the corresponding conjugate of $\gC$. One can express the induction of $\mu\in\irs(\gC)$ to $\irs(G)$ by:

$$
 \text{Ind}_\gC^G(\mu)=\frac{1}{m(\gO)}\int_{\gO}(i_g)_*\mu~dm(g),
$$
where $m$ is Haar measure on $G$, $\gO$ is a fundamental domain for $G/\gC$ and $i_g:\sub_G\to\sub_G$ corresponds to conjugation by $g$ .

\subsection{Connection with pmp actions}
Let $G\act (X,m)$ be a probability measure preserving action. By a result of Varadarajan, the stabilizer of almost every point in $X$ is closed in $G$. Moreover, the stabilizer map $X\to \text{Sub}_G,~x\mapsto G_x$ 
is measurable, and hence one can push the measure $m$ to an IRS on $G$. In other words the random subgroup is the stabilizer of a random point in $X$.

This reflects the connection between invariant random subgroups and pmp actions. Moreover, the study of pmp $G$-spaces can be divided to 
\begin{itemize}
\item the study of stabilizers (i.e. IRS), 
\item the study of orbit spaces
\end{itemize}
and the interplay between the two. 

\medskip

The connection between IRS and pmp actions goes also in the other direction:

\begin{thm}\label{thm:satbilizers}
Let $G$ be a locally compact group and $\mu$ an IRS in $G$. Then there is a probability space $(X,m)$ and a measure preserving action $G\act X$ such that $\mu$ is the push-forward of the stabilizer map $X\to\text{Sub}_G$.
\end{thm}

This was proved in \cite{AGV} for discrete groups and in \cite[Theorem 2.4]{7} for general $G$. The first thing that comes to mind is to take the given $G$ action on $(\text{Sub}_G,\mu)$, but then the stabilizer of a point $H\in\text{Sub}_G$ is $N_G(H)$ rather than $H$. To correct this one can consider the larger space $\text{Cos}_G$ of all cosets of all closed subgroups, as a measurable $G$-bundle over $\text{Sub}_G$. Defining an appropriate invariant measure on $\text{Cos}_G\times\BR$ and replacing each fiber by a Poisson process on it, gives the desired probability space.

\subsection{Topology}
We shall denote by $\text{IRS}(G)$ the space $G$-invariant probability measures on $\sub(G)$
$$
 \text{IRS}(G):=\text{Prob}(\text{Sub}_G)^G
$$
equipped with the $w^*$-topology. By Alaoglu's theorem $\text{IRS}(G)$ is compact.

\subsection{Existence}

An interesting yet open question is whether this space is always non-trivial.

\begin{ques}\label{ques:existance}
Does every non-discrete locally compact group admit a non-trivial IRS?
\end{ques}

A counterexample, if exists, should in particular be a simple group without lattices. Currently the only known such example is the Neretin group and some close relatives (see \cite{BCGM}).
The question whether the Neretin group admits non-trivial IRS has two natural sub-questions:

\begin{ques}
\begin{enumerate}
\item Does the Neretin group admit a (non-discrete) closed subgroup of finite co-volume?
\item Does the Neretin group admit a non-trivial discrete IRS, i.e. an IRS with respect to which a random subgroup is a.s. discrete? 
\end{enumerate}
\end{ques}

\begin{rem}
There are many discrete groups without nontrivial IRS, for instance $\PSL_n(\BQ)$, and also the Tarski Monsters.
\end{rem}

\subsection{Soficity of IRS}

\begin{definition}
Let us say that an IRS $\mu$ is {\it co-sofic} if it is a weak-$*$ limit in $\text{IRS}(G)$ of ones supported on lattices.
\end{definition}

The following question can be asked for any locally compact group $G$, however I find the $3$ special cases of $G=$\blue{$SL_2(\RR),SL_2(\QQ_p)$} and \blue{$\text{Aut}(T)$} particularly intriguing:

\begin{ques}
Is every IRS in $G$ co-sofic?  
\end{ques}

\begin{exercise}
1. Show that the case $G=F_n$, the discrete rank $n$ free group, is equivalent to the Aldous--Lyons conjecture that every unimodular network (supported on rank $n$ Schreier graphs) is a limit of ones corresponding to finite Schreier graphs \cite{AL}.\\

2. A Dirac mass $\delta_N,~N\lhd F_n$ is co-sofic iff the corresponding group $G=F_n/N$ is sofic.
\end{exercise}

\section{IRS and lattices}
Viewing IRS as a generalization of lattices there are two directions toward which one is urged to go:
\begin{enumerate}
\item Extend classical theorems about lattices to general IRS.
\item Use the compact space IRS$(G)$ in order to study its special `lattice' points.
\end{enumerate}

Remarkably, the approach $(2)$ turns out to be quite fruitful in the theory of asymptotic properties of lattices. We shall see later on (see Section \ref{section:HR}) an example of how rigidity properties of $G$-actions yield interesting data of the geometric structure of locally symmetric spaces $\gC\backslash G/K$ when the volume tends to infinity. This approach is also useful for proving uniform statements regurding the set of all lattices (see Section \ref{sec:KM}). 

%Here is another direction in the spirit of $(2)$, this time with a fixed volume: 

The following sections will demonstrate in various forms both approaches (1) and (2) and some interactions between the two (for instance we will use the extension of Borel's density theorem for IRS in order to prove a strong version of the Kazhdan--Margulis theorem about lattices).

\subsection{The IRS compactification of moduli spaces}
One direction in the spirit of $(2)$ which hasn't been applied yet (to the author's knowledge) is simply to obtain, using $\irs(G)$, new compactifications of certain natural spaces. 

\begin{exam}
Let $\gS$ be a closed surface of genus $\ge 2$. As we have seen in \ref{subsection:examples}(2), every hyperbolic structure on $\gS$ corresponds to an IRS in $\PSL_2(\BR)$. Taking the closure in IRS$(G)$ of the set of hyperbolic structures on $\gC$, one obtains an interesting compactification of the moduli space of $\gS$.  

\begin{prob}
Analyse the IRS compactification of $\text{Mod}(\gS)$.
\end{prob} 

Note that the resulting compactification is similar to (but is not exactly) the Deligne--Munford compactification.

\end{exam}

%Let us now describe two important results in the spirit of $(1)$: 

\subsection{Borel density theorem}

\begin{thm}[Borel density theorem for IRS, \cite{7}]\label{thm:BDT}
Let $G$ be a connected non-compact simple (center-free) Lie group. Let $\mu$ be an IRS on $G$ without atoms. Then a random subgroup is $\mu$-a.s. discrete and Zariski dense.  
 \end{thm}

%The idea is to consider the maps $\text{Sub}_G\to Gr(\text{Lie}(G))$
%$$
% H\mapsto \text{Lie}(H),~\text{and} ~H\mapsto \text{Lie}(\overline{H}^Z),
%$$ 
%and push the invariant measure to one on ${Gr}(\text{Lie}(G))$. By Furstenberg's lemma every such measure is trivial.
%
%\medskip

Note that since $G$ is simple, the the only possible atoms are at the trivial group $\{ 1\}$ and at $G$. Since $G$ is an isolated point in $\text{Sub}_G$, it follows that
%the set of IRS which are a.s. discrete (we shall call such --- discrete IRS) is a compact space. We shall denote this space by:
$$
 \irs_d(G):=\{\mu\in\text{IRS}(G):\text{a}~\mu\text{-random subgroup is a.s. discrete}\}
$$  
is a compact space. We shall refer to the points of $\irs_d$ as discrete IRS.

\medskip

In order to prove Theorem \ref{thm:BDT} one first observes that there are only countably many conjugacy classes of non-trivial finite subgroups in $G$, hence the measure of their union is zero with respect to any non-atomic IRS. Then one can apply the same idea as in Furstenberg's proof of the classical Borel density theorem \cite{Fur}. Indeed, taking the Lie algebra of $H\in\text{sub}_G$ as well as of its Zariski closure induce measurable maps (see \cite{GL})
$$
 H\mapsto \text{Lie}(H),~H\mapsto\text{Lie}(\overline{H}^Z)
$$
As G is noncompact, Furstenberg's argument implies that the Grassman variety of non-trivial subspaces of $\text{Lie}(G) $ does not carry an $\Ad(G)$-invariant measure. It follows that $ \text{Lie}(H)= 0$ and $\text{Lie}(\overline{H}^Z)=
\text{Lie}(G)$
almost surely, and the two statements of the theorem follow.

%and since a non-trivial Grassman variety associated to subspaces of Lie$(G)$ does not carry an Ad$(G)$-invariant measure (because $G$ is noncompact) one deduces that $\text{Lie}(H)=0$, i.e. that $H$ is discrete, and $\text{Lie}(\overline{H}^Z)=\text{Lie}(G)$ almost surely (see \cite{7} for more details).

\begin{rem}
The analog of Theorem \ref{thm:BDT} holds, more generally, in the context of semisimple analytic groups over local fields, see \cite{GL}.
\end{rem}

\subsection{Kazhdan--Margulis theorem}\label{sec:KM}

%\begin{prop}\label{prop:BD}
%Suppose that $G$ is a connected non-compact simple Lie group.

\begin{defn}
A family $\mathcal{F}\subset\irs(G)$ of invariant random subgroups is said to be \it{weakly uniformly discrete} if for every $\gep>0$ there is an identity neighbourhood $U\subset G$ such that
$$
 \mu (\{\gC\in\sub_G:\gC\cap U\ne\{1\}\})<\gep
$$
for every $\mu\in \mathcal{F}$.
\end{defn}

\begin{thm}\label{thm:WUD}
Let $G$ be a connected non-compact simple Lie group. Then $\irs_d(G)$ is weakly uniformly discrete.
\end{thm}

Assume, in contrary, that for some $\gep>0$ we can find for every identity neighbourhood $U\subset G$ a discrete IRS $\mu_U$ such that 
$$
 \mu_U\big( \{\gC\in\sub_G:\gC\cap U~\text{non-trivially}\}\big)\ge\gep.
$$
Then letting $U$ run over a suitable base of identity neighbourhoods and taking a weak limit $\mu$, it would follow that $\mu$ is not discrete in contrast to the compactness of $\irs_d(G)$.

\medskip 

As a straight forward consequence, taking $\gep<1$, we deduce the classical Kazhdan--Margulis theorem, and in particular the positivity of the lower bound on the volume of locally $G/K$-manifolds:

\begin{cor}[Kazhdan--Margulis theorem \cite{KM}]
There is an identity neighbourhood $\gO\subset G$ such that for every lattice $\gC\le G$ there is $g\in G$ such that $g\gC g^{-1}\cap \gO=\{1\}$.
\end{cor} 

Viewing the stabilizer of a random point in a probability measure preserving $G$-space as an IRS, Theorem \ref{thm:WUD} can be reformulated as follows:

\begin{thm}[p.m.p actions are uniformly weakly locally free]
For every $\gep>0$ there is an identity neighbourhood $U\subset G$ such that for every non-trivial ergodic p.m.p $G$-space $(X,m)$ there is a subset $Y\subset X$ with $m(Y)>1-\gep$ such that $u\cdot y\ne y$ for all $y\in Y$ and $u\in U$.
\end{thm}

\medskip

For complete proofs of the results of this subsection, in a more general setup, see \cite{WUD}.

\subsection{Stuck--Zimmer rigidity theorem}

Perhaps the first result about IRS and certainly one of the most remarkable, is the Stuck--Zimmer rigidity theorem, which can be regarded as a (far reaching) generalisation of Margulis' normal subgroup theorem.

\begin{thm}[Stuck--Zimmer \cite{SZ}\label{thm:SZ}]
Let $G$ be a connected simple Lie group of real rank $\ge 2$. Then every ergodic pmp action of $G$ is either (essentially) free or transitive.
\end{thm}

In view of Theorem \ref{thm:satbilizers}, one can read Theorem \ref{thm:SZ} as: {\it every non-atomic ergodic IRS in $G$ is of the form $\mu_\gC$ for some lattice $\gC\le G$}.

\medskip
%\begin{exercise}
%Using the `IRS induction' described in Paragraph \ref{subsection:examples}(3), show that Theorem \ref{thm:SZ} implies the Margulis normal subgroup theorem, namely that a normal subgroup $N\lhd \gC$ of a lattice $\gC\le G$ is of finite index.
%\end{exercise}

In order to see that Margulis normal subgroup theorem is a consequence of Theorem \ref{thm:SZ}, suppose that $\gC\le G$ is a lattice and $N\lhd\gC$ is a normal subgroup in $\gC$. As described in Example \ref{subsection:examples}$(3)$ there is an IRS $\mu$ supported on (the closure of) the conjugacy class of $N$. Since $G\act G/\gC$ is ergodic and $(\text{Sub}_G,\mu)$ is a factor, it is also ergodic. In view of Theorem \ref{thm:satbilizers} there is an ergodic pmp $G$-space $(X,m)$ such that the stabilizer of an $m$-random point is a $\mu$-random conjugate of $N$. By Theorem \ref{thm:SZ}, $X$ is transitive, i.e. $X=G/N$. It follows that $G/N$ carries a $G$-invariant probability measure. Thus $N$ is a lattice, hence of finite index in $\gC$.

\begin{rem}
\begin{enumerate}
\item Stuck and Zimmer proved the theorem for the wider class of higher rank semisimple groups with property $(T)$, where in this case the assumption is that every factor is noncompact and acts ergodically. The situation for certain groups, such as $\SL_2(\BR)\times\SL_2(\BR)$ is still unknown.

\item Recently A. Levit \cite{Levit} proved the analog result for analytic groups over non-archimedean local fields. 
\end{enumerate}
\end{rem}

\medskip

\subsection{An exotic IRS in rank one}

In the lack of Margulis' normal subgroup theorem there are IRS supported on non-lattices. Indeed, if $G$ has a lattice with an infinite index normal subgroup $N\lhd\gC$, arguing as in the previous section, one obtains an ergodic p.m.p. space for which almost any stabilizer  is a conjugate of $N$. 

We shall now give a more interesting example, in the lack of rigidity:

\begin{exam}[An exotic IRS in $\PSL_2(\BR)$, \cite{7}]
Let $A,B$ be two copies of a torus with 2 open discs removed. We choose hyperbolic metrics on $A$ and $B$ so that all $4$ boundary components are geodesic circles of the same length, and such that $A$ admits a closed geodesic of length much smaller than the injectivity radius at any point of $B$. We may agree that one boundary component of $A$ (resp. of $B$) is `on the left side' and the other is `on the right side', and fix a special point on each boundary component, in order to specify a gluing pattern of a `left' copy and a `right' copy, each of either $A$ or $B$.\footnote{All the nice figures in these lecture notes were made by Gil Goffer}

%%%%
%%%%
%%%%

\begin{figure}[h]
    \centering
    \includegraphics[width=0.8\textwidth]{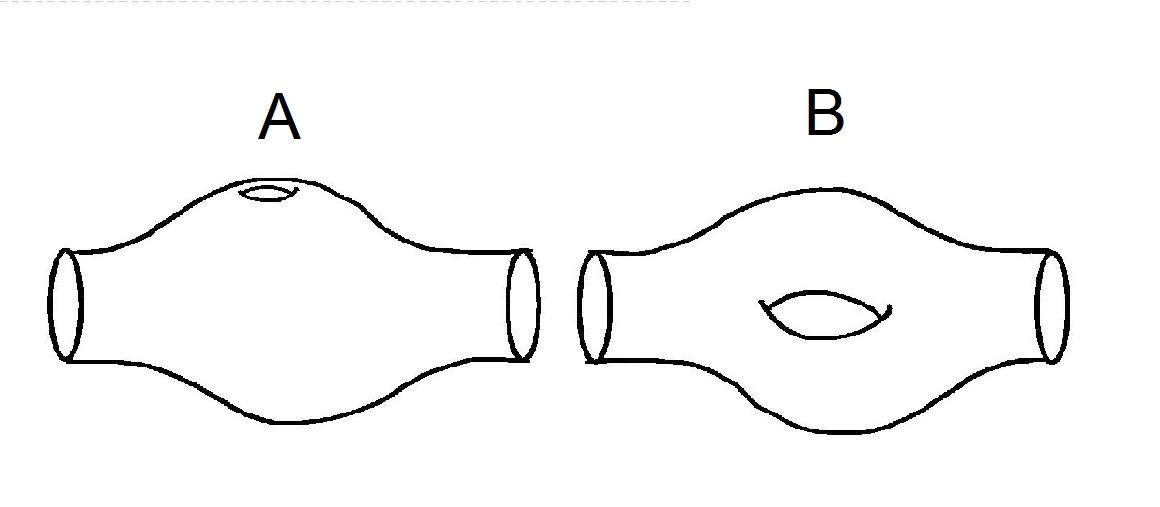}
    \caption{The hyperbolic building blocks}
    \label{fig:awesome_image}
\end{figure}

%%%%
%%%%
%%%%

Now consider the space $\{A,B\}^\BZ$ with the Bernoulli measure $(\frac{1}{2},{1\over 2})^\BZ$. Any element $\ga$ in this space is a two sided infinite sequence of $A$'s and $B$'s and we can glue copies of $A,B$ `along a bi-infinite line' following this sequence. This produces a random surface $M^\ga$. Choosing a probability measure on the unit tangent bundle of $A$ (resp. of $B$) we define an IRS in $\PSL_2(\BR)$ as follows. First choose $M^\ga$ randomly, next choose a point and a unit tangent vector in the copy of $A$ or $B$ which lies at the place $M^\ga_0$ (above $0$ in the sequence $\ga$), then take the fundamental group of $M^\ga$ according to the chosen point and direction.

%%%%
%%%%
%%%%

\begin{figure}[h]
    \centering
    \includegraphics[width=0.8\textwidth]{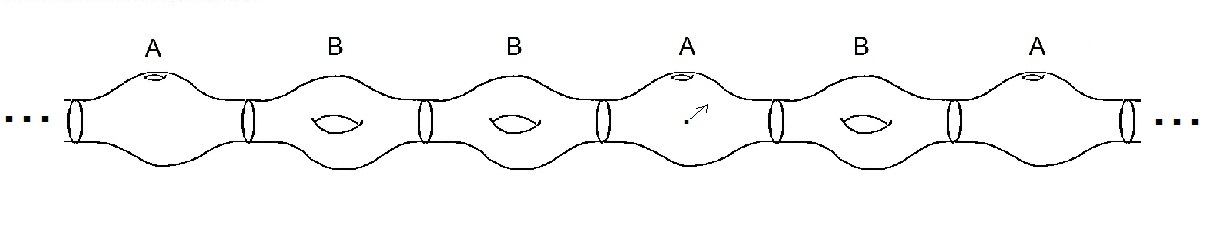}
    \caption{A random surface}
    \label{fig:awesome_image}
\end{figure}

%%%%
%%%%
%%%%

As the $\BZ$-action on the Bernoulli space of sequences is ergodic one sees that the corresponding IRS is also ergodic. It can be shown that almost surely the corresponding group is not contained in a lattice in $\PSL_2(\BR)$. However, this IRS is co-sofic. Analog constructions can be made in $\SO(n,1)$ for all $n$'s (see \cite[Section 13]{7}).
\end{exam}

%
%
%
%
%
%
%The following question can be asked for any locally compact group $G$, however I find the $3$ special cases of $G=\SL_2(\BR),\SL_2(\BQ_p)$ and $\Aut(T)$ particularly intriguing:
%
%\begin{ques}
%Is every discrete IRS a weak-$*$ limit of ones supported on lattices?  
%\end{ques}
%
%In view of the connection between IRS's and the Benjamini--Schramm topology which we are about to describe, the case of $G=\Aut(T)$ is strongly related to the famous Aldous--Lyons conjecture that every unimodular network is sofic \cite{AL} (i.e. a limit of ones corresponding to finite graphs).
%
%\begin{exercise}
%Show that the same question for $G=F_n$, the discrete rank $n$ free group, is equivalent to the question whether every $n$-generated group is sofic.
%\end{exercise}
%
\section{The Gromov--Hausdorff topology}

Given a compact metric space $X$, the Hausdorff distance $\text{Hd}_X(A,B)$ between two closed subsets is defined as
$$
 \text{Hd}_X(A,B):=\inf\{\gep: N_\gep(A)\supset B~\text{and}~N_\gep(B)\supset A\},
$$
where $N_\gep(A)$ is the $\gep$-neighborhood of $A$. The space $2^X$ of closed subsets of $X$ equipped with the Hausdorff metric, is compact.

Given two compact metric spaces $X,Y$, the Gromov distance $\text{Gd}(X,Y)$ is defined as
$$
 \text{Gd}(X,Y):=\inf_Z\{\text{Hd}_Z(i(X),j(Y))\},
$$ 
over all compact metric spaces $Z$ admitting isometric copies $i(X),j(Y)$ of $X,Y$ respectively. 
If $(X,p),(Y,q)$ are pointed compact metric spaces, i.e. ones with a chosen point, we define the Gromov distance 
$$
 \text{Gd}((X,p),(Y,q)):=\inf_Z\{\text{Hd}_Z(i(X),j(Y))+d_Z(i(p),j(q))\}.
$$ 

The Gromov--Hausdorff distance between two pointed proper (not necessarily bounded) metric spaces $(X,p),(Y,q)$ can be defined as
$$
 \text{GHd}((X,p),(Y,q)):=\sum_{n\in\BN} {1\over 2^n}\text{Gd}((B_X(n),p),(B_Y(n),q)),
$$ 
where $B_X(n)$ is the ball of radius $n$ around $p$.

\section{The Benjamini--Schramm topology}

Let $\mathcal{M}$ be the be the space of all (isometry classes of) pointed proper metric spaces equipped with the Gromov--Hausdorff topology. This is a huge space and for many applications it is enough to consider compact subspaces of it obtained by bounding the geometry. That is, let $f(\gep,r)$ be an integer valued function defined on $(0,1)\times\BR^{>0}$, and let $\mathcal{M}_f$ consist of those spaces for which $\forall \gep,r$, the $\gep$-entropy of the $r$-ball 
$B_X(r,p)$ around the special point is bounded by $f(\gep,r)$, i.e. no $f(\gep,r)+1$ points in $B_X(r,p)$ form an $\gep$-discrete set. Then $\mathcal{M}_f$ is a compact subspace of $\mathcal{M}$.   

In many situations one prefers to consider some variants of $\mathcal{M}$ which carry more information about the spaces. 
For instance when considering graphs, it may be useful to add colors and orientations to the edges. The Gromov--Hausdorff distance defined on these objects should take into account the coloring and orientation.
%For instance when considering graphs, it may be useful to add colors and directions to the edges, and the distance between rooted coloured graphs remembers the coloring. 
Another example is smooth Riemannian manifolds, in which case it is better to consider framed manifolds, i.e. manifold with a chosen point and a chosen frame at the tangent space at that point. In that case, one replace the Gromov--Hausdorff topology by the ones determined by $(\gep,r)$ relations (see \cite[Section 3]{7} for details), which remembers also the directions from the special point.

We define the \blue{Benjamini--Schramm} space $\mathcal{BS}=\text{Prob}(\mathcal{M})$ to be the space of all Borel probability measures on $\mathcal{M}$ equipped with the weak-$*$ topology. Given $f$ as above, we set $\mathcal{BS}_f:=\text{Prob}(\mathcal{M}_f)$. Note that $\mathcal{BS}_f$ is compact.

The name of the space is chosen to hint that this is the same topology induced by `local convergence', introduced by Benjamini and Schramm in \cite{BS}, when restricting to measures on rooted graphs. Recall that a sequence of random rooted bounded degree graphs converges to a limiting distribution iff for every $n$ the statistics of the $n$ ball around the root (i.e. the probability vector corresponding to the finitely many possibilities for $n$-balls) converges to the limit. 

The case of general proper metric spaces can be described similarly. A sequence $\mu_n\in\mathcal{BS}_f$ converges to a limit $\mu$ iff for any compact pointed `test-space' $M\in\mathcal{M}$, any $r$ and arbitrarily small\footnote{This doesn't mean that it happens for all $\gep$.} $\gep>0$, the $\mu_n$ probability that the $r$ ball around the special point is `$\gep$-close' to $M$ tends to the $\mu$-probability of the same event.

\begin{exam}
An example of a point in $\mathcal{BS}$ is a measured metric space, i.e. a metric space with a Borel probability measure. 
A particular case is a finite volume Riemannian manifold --- in which case we scale the Riemannian measure to be one, and then randomly choose a point (and a frame).
\end{exam}

Thus a finite volume locally symmetric space $M=\gC\backslash G/K$ produces both a point in the Benjamini--Schramm space and an IRS in $G$. This is a special case of a general analogy that I'll now describe. Given a symmetric space $X$, let us denote by $\mathcal{M}(X)$ the space of all pointed (or framed) complete Riemannian orbifolds whose universal cover is $X$, and by
$\mathcal{BS}(X)=\text{Prob}(\mathcal{M}(X))$ the corresponding subspace of the Benjamini--Schramm space. 

Let $G$ be a non-compact simple Lie group with maximal compact subgroup $K\le G$ and an associated Riemannian symmetric space $X=G/K$. There is a natural map 
$$
 \{\text{discrete subgroups of }~G\}\to \mathcal{M}(X),~\gC\mapsto \gC\backslash X.
$$  
It can be shown that this map is continuous, hence inducing a continuous map
$$
 \irs_d(G)\to \mathcal{BS}(X).
$$
It can be shown that the later map is one to one, and since $\irs_d(G)$ is compact, it is
a homeomorphism to its image (see \cite[Corollary 3.4]{7}). 
%Thus the IRS topology coincides with the BS-topology if we identify $\irs_d(G)$ with its image.

\begin{exercise}[Invariance under the geodesic flow]
Given a tangent vector $\overline{v}$ at the origin (the point corresponding to $K$) of $X=G/K$, define a map 
$\mathcal{F}_{\overline{v}}$ from $\mathcal{M}(X)$ to itself by moving the special point using the exponent of 
$\overline{v}$ and applying parallel transport to the frame. This induces a homeomorphism of $\mathcal{BS}(X)$. Show that the image of $\irs_d(G)$ under the map above is exactly the set of $\mu\in \mathcal{BS}(X)$ which are invariant under $\mathcal{F}_{\overline{v}}$ for all $\overline{v}\in T_K(G/K)$.
\end{exercise}

Thus we can view geodesic-flow invariant probability measures on framed locally $X$-manifolds as IRS on $G$ and vice versa, and the Benjamini--Schramm topology on the first coincides with the IRS-topology on the second.

\begin{exercise}
Show that the analogy above can be generalised, to some extent, to the context of general locally compact groups. Given a locally compact group $G$, fixing a right invariant metric on $G$, we obtain a map $\text{Sub}_G\to\mathcal{M},~H\mapsto G/H$, where the metric on $G/H$ is the induced one. Show that this map is continuous and deduce that it defines a continuous map
$\text{IRS}(G)\to \mathcal{BS}$.
\end{exercise}

For the sake of simplicity we shall now forget `the frame' and consider pointed $X$-manifolds, and $\mathcal{BS}(X)$ as probability measures on such. We note that while for general Riemannian manifolds there is a benefit for working with framed manifolds, in the world of locally symmetric spaces of non-compact type, pointed manifolds, and measures on such, behave nicely enough.
 
In order to examine convergence in $\mathcal{BS}(X)$ it is enough to use as `test-space' balls in locally $X$-manifolds. Moreover, since $X$ is non-positively curved, a ball in an $X$-manifold is isometric to a ball in $X$ iff it is contractible. 

For an $X$-manifold $M$ and $r>0$, we denote by $M_{\ge r}$ the $r$-thick part in $M$:
$$
 M_{\ge r}:=\{x\in M:\text{InjRad}_M(x)\ge r\},
$$
where $\text{InjRad}_M(x)=\sup\{\gep:B_M(x,\gep)~\text{is contractible}\}$. 

Note that since $X$ is a homogeneous space, all choices of a probability measure on $X$ correspond to the same point in $\mathcal{BS}(X)$, and we shall denote this point by $X$, with a slight abuse of notations. We have the following simple characterisation of convergence to $X$:

\begin{prop}
A sequence $M_n$ of finite volume $X$-manifolds BS-converges to $X$ iff
$$
 \frac{\vol((M_n)_{\ge r})}{\vol(M_n)}\to 1,
$$
for every $r>0$.
\end{prop}

\section{Higher rank and rigidity}\label{section:HR} 

Suppose now that $G$ is a non-compact simple Lie group of real rank at least $2$. The following result from \cite{7} can be interpreted as `large manifolds are almost everywhere fat':

\begin{thm}[\cite{7}]\label{thm:main}
Let $M_n=\gC_n\backslash X$ be a sequence of finite volume $X$-manifolds with $\vol(M_n)\to\infty$. Then $M_n\to X$ in the Benjamini--Schramm topology.
\end{thm} 

This means that for any $r$ and $\gep$ there is $V(r,\gep)$ such that if $M$ is an $X$-manifold of volume $v\ge V(r,\gep)$ then $\frac{\vol(M_{\ge r})}{v}\ge 1-\gep$ (see Figure \ref{whale}).

\begin{figure}[h]
    \centering
    \includegraphics[width=0.8\textwidth]{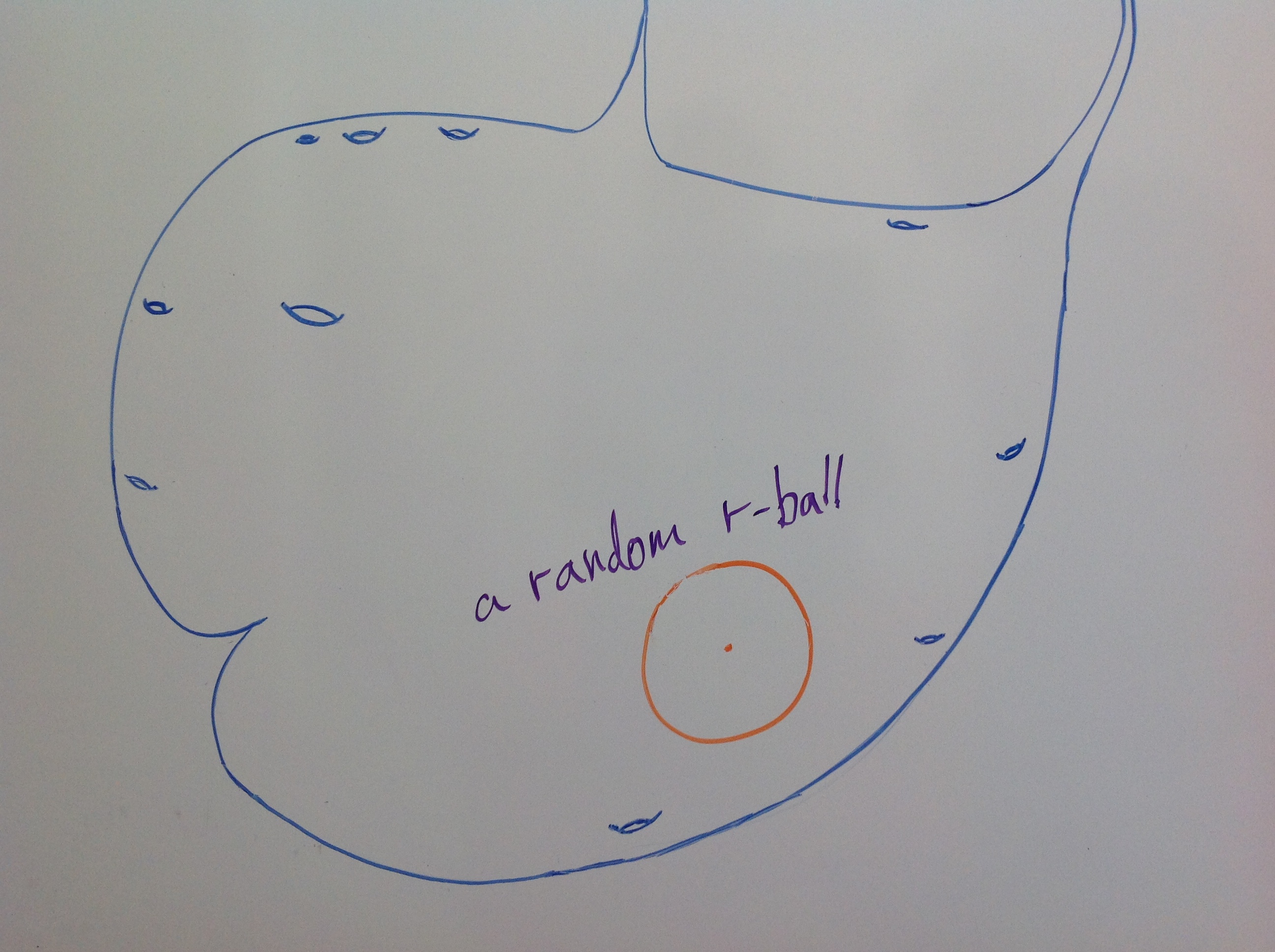}
    \caption{A large volume manifold}
    \label{whale}
\end{figure}

Using the dictionary from the previous section we may reformulate Theorem \ref{thm:main} in the language of IRS:

\begin{thm}[\cite{7}]\label{thm:main-IRS}
Let $\gC_n\le G$ be a sequence of lattices with $\vol(G/\gC_n)\to\infty$ and denote by $\mu_n$ the corresponding IRS. Then $\mu_n\to \gd_{\{1\}}$.
\end{thm} 

The proof makes use of the equivalence between the two formulations. The main ingredients in the proof are the Stuck--Zimmer rigidity theorem \ref{thm:SZ} and Kazhdan's property (T), which will be used at two places (in addition property (T) is used in the proof of \ref{thm:SZ}). 

Recall that by Kazhdan's theorem, $G$ has property (T). This implies that a limit of ergodic measures is ergodic:

\begin{thm}[\cite{GW}]\label{thm:GW}
Let $G$ be a group with property (T) acting by homeomorphisms on a compact Hausdorff space $X$. Then the set of ergodic $G$-invariant probability Borel measures on $X$ is $w^*$-closed.
\end{thm}   

The idea is that if $\mu_n$ are probability measures converging to a limit $\mu$ and $\mu$ is not ergodic, then there is a continuous function $f$ on $X$ which, as a function in $L_2(\mu)$ is $G$-invariant, orthogonal to the constants and with norm $1$. Thus for large $n$ we have that $f$ is almost invariant in $L_2(\mu_n)$, almost orthogonal to the constants and with norm almost $1$. Since $G$ has property (T) it follows that there is an invariant $L_2(\mu_n)$ function close to $f$, so $\mu_n$ cannot be ergodic.  

\medskip

Let now $\mu_n$ be a sequence as in \ref{thm:main-IRS}, and let $\mu$ be a weak-$*$ limit of $\mu_n$. Our aim is to show that $\mu=\gd_{\{1\}}$. Up to replacing $\mu_n$ by a subnet, we may suppose that $\mu_n\to \mu$. Let $M_n=\gC_n\backslash X$ be the corresponding manifolds, as in \ref{thm:main}. By Theorem \ref{thm:GW} we know that $\mu$ is ergodic. The following result is a consequence of Theorem \ref{thm:SZ}:

\begin{prop}
The only ergodic IRS on $G$ are $\gd_G,\gd_{\{1\}}$ and $\mu_\gC$ for $\gC\le G$ a lattice. 
\end{prop}

\begin{proof}
Let $\mu$ be an ergodic IRS on $G$. By Theorem \ref{thm:satbilizers} $\mu$ is the stabilizer of some pmp action $G\act (X,m)$. By Theorem \ref{thm:SZ} the latter action is either essentially free, in which case $\mu=\gd_{\{1\}}$, or transitive, in which case the (random) stabilizer is a subgroup of co-finite volume. The Borel density theorem implies that in the latter case, the stabilizer is either $G$ or a lattice $\gC\le G$.
\end{proof}

Thus, in order to prove Theorem \ref{thm:main-IRS} we have to exclude the cases $\mu=\gd_G$ and $\mu=\mu_\gC$. The case $\mu=\gd_G$ is impossible since $G$ is an isolated point in $\text{Sub}_G$ (see \ref{prop:G-isolated}). Let us now suppose that $\mu=\mu_\gC$ for some lattice $\gC\le G$ and aim towards a contradiction. For this, we will adopt the formulation of \ref{thm:main}. Thus we suppose that $M_n\to M=\gC \backslash X$.

Recall that Property (T) of $G$ implies that there is a lower bound $C>0$ for the Cheeger constant of all finite volume $X$-manifolds. For our purposes, the Cheeger constant of a manifold $M$ can be defined as the infimum of 
$$
 \frac{\vol (N_1(S))}{\min \{\vol(M_i)\}},
$$
where $S$ is a subset which disconnects the manifold, $N_1(S)$ is its $1$-neighbourhood, and $M_i$ is chosen from the connected pieces of $M\setminus S$. 

Since $M=\gC\backslash X$ has finite volume we may pick a point $p\in M$ and $r$ large enough so that the volume of $B_M(p,r-1)$, the $r-1$ ball around $p$ in $M$, 
is greater than $\vol(M)(1-C)$ (note that if $M$ is compact we may even take a ball that covers $M$). In particular, when taking $S=\{x\in M:d(x,p)=r\}$ we have that $\frac{\vol(N_1(S))}{\vol(B_M(p,r-1))}<C$. This on itself does not contradict property (T) since the complement of $B_M(p,r+1)$ is very small. 

Now since $M_n$ converges to $M$ in the BS-topology, it follows that for large $n$, a random point $q$ in $M_n$ with positive probability satisfies that 
$$
 \frac{\vol(B_{M_n}(q,r+1)\setminus B_{M_n}(q,r-1))}{\vol(B_{M_n}(q,r-1))}<C.
$$ 
Bearing in mind that $\vol(M_n)\to\infty$, we get that for large $n$, the complement $M_n\setminus B_{M_n}(q,r+1)$ has arbitrarily large volume, and in particular greater than $\vol(B_{M_n}(q,r-1))$, see Figure \ref{fig:Cheeger}. Now, this contradicts the assumption that $C$ is the Cheeger constant of $X$.\qed

%%%
%%%
%%%

\begin{figure}[h]\label{fig:Cheeger}
    \centering
    \includegraphics[width=0.8\textwidth]{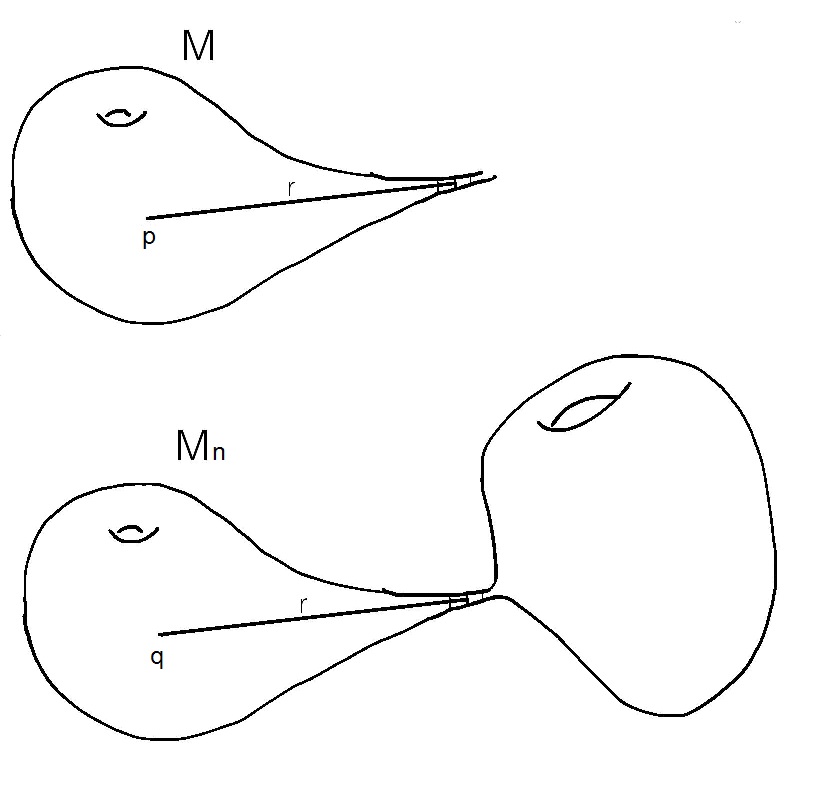}
    \caption{The Cheeger constant of $M_n$ is too small.}
    \label{whale}
\end{figure}

%%%
%%%
%%%

\medskip

Note that Theorems \ref{thm:main} and \ref{thm:main-IRS} can also be formulated as:
\begin{itemize}
\item The set of extreme points in IRS(G) (the ergodic IRS) is closed and equals $\{\gd_G,\gd_{\{1\}},\mu_\gC, \gC\le G~\text{a lattice}\}$ and its unique accumulation point is $\gd_{\{1\}}$,
\end{itemize}

or as:
\begin{itemize}
\item The space of geodesic flow invariant probability measures on $\mathcal{M}(X)$ is compact and convex, its extreme points are the finite volume $X$-manifolds and the space $X$, and the later being the only accumulation point. 
\end{itemize}

Finally let us note that Theorem \ref{thm:main} has many applications in the theory of asymptotic invariants, and in particular $L_2$-invariant, of arithmetic groups and locally symmetric manifolds. Most of the work \cite{7} is dedicated to such asymptotic results and our main new ingredient is Theorem \ref{thm:main}.
For instance, one quite immediate application is that if $M_n$ is a sequence of uniformly discrete (i.e. with a uniform lower bound on the injectivity radius) $X$-manifolds with volume $\vol(M_n)\to\infty$ then the normalised betti numbers converge to a limit
$$
 \frac{b_i(M_n)}{\vol(M_n)}\to b_i^{(2)}(X),
$$
and the limit $b_i^{(2)}(X)$ is computable, and vanishes for $i\ne\dim(X)/2$ (cf. \cite{7-note}).

\end{document}